\documentclass{article}

\usepackage{amssymb,amsmath,amsthm}
\usepackage{fullpage}
\usepackage{mathtools} 

\newcommand{\be}{\begin{enumerate}}
\newcommand{\ee}{\end{enumerate}}
\newcommand{\beas}{\begin{eqnarray*}}
\newcommand{\eeas}{\end{eqnarray*}}
\newcommand{\bea}{\begin{eqnarray}}
\newcommand{\eea}{\end{eqnarray}}
\newcommand{\beq}{\begin{equation}}
\newcommand{\eeq}{\end{equation}}
\newcommand{\anbc}[2]{\genfrac{\langle}{\rangle}{0pt}{}{#1}{#2}}
\newcommand{\qbc}[2]{\genfrac{[}{]}{0pt}{}{#1}{#2}}

\newcommand{\modd}[1]{\,(\mathrm{mod}\, #1)}
\newcommand{\mo}{\mathrm{mo}}
\newcommand{\mmp}{\mathrm{mmp}}
\newcommand{\rmp}{\mathrm{rmp}}
\newcommand{\spread}{\mathrm{spread}}

\theoremstyle{theorem}
\newtheorem{theorem}{Theorem}
\newtheorem*{conjecture}{Conjecture}
\theoremstyle{definition}

\newtheorem*{example}{Example}
\newtheorem{remark}{Remark}

\def\rr{\mathbb{R}}

\def\nn{\mathbb{N}}
\def\cc{\mathbb{C}}

\begin{document}
\title{Some Linear Recurrences Motivated by Stern's Diatomic Array } 
\markright{Some Linear Recurrences}

\date{\today}

\author{Richard P. Stanley}

\maketitle

\begin{abstract}
We define a triangular array closely related to Stern's diatomic array
and show that for a fixed integer $r\geq 1$, the sum $u_r(n)$ of the
$r$th powers of the entries in row $n$ satisfy a linear recurrence
with constant coefficients. The proof technique yields a vast
generalization. In certain cases we can be more explicit about the
resulting linear recurrence.
\end{abstract}

\section{Introduction.}\label{sec1}
We first define an array of numbers analogous to Pascal's triangle (or
the arithmetic triangle). The rows will be numbered $0,1,\dots$. The
first row consists of a single 1, and every subsequent row begins and
ends with a 1, just like Pascal's triangle. We add two consecutive
numbers in row $n$ and place the sum in row $n+1$ to the right of the
first number and left of the second number, again just like Pascal's
triangle. However, we also bring down (copy) into row $n+1$ each entry
in row $n$, placing it directly below. The first four rows look like
   $$ \begin{array}{ccccccccccccccc}
     & & & & & & & 1\\ & & & 1 & & & & 1 & & & & 1\\ 
     & 1 & & 1 & & 2 & & 1 & & 2 & & 1 & & 1\\
     1 & 1 & 2 & 1 & 3 & 2 & 3 & 1 & 3 & 2 & 3 & 1 & 2 & 1 & 1\\
     & & & & & & & \vdots \end{array}. $$
We will call this array \emph{Stern's triangle}, though ``triangle''
is somewhat of a misnomer since the number of entries in each row
grows exponentially, not linearly.

Stern's triangle is closely related to a well-known array called
\emph{Stern's diatomic array} \cite{stern}. It has the same recursive
rules as Stern's triangle, but the first row consists of two 1's which
are brought down to form the first and last element of each row. It
looks like
   $$ \begin{array}{ccccccccccccccccc}
     1 &   &   &   &   &   &   &   &   &   &   &   &   &   &   &   & 1\\
     1 &   &   &   &   &   &   &   & 2 &   &   &   &   &   &   &   & 1\\
     1 &   &   &   & 3 &   &   &   & 2 &   &   &   & 3 &   &   &   & 1\\
     1 &   & 4 &   & 3 &   & 5 &   & 2 &   & 5 &   & 3 &   & 4 &   & 1\\ 
     1 & 5 & 4 & 7 & 3 & 8 & 5 & 7 & 2 & 7 & 5 & 8 & 3 & 7 & 4 & 5 & 1\\
       &   &   &   &   &   &   &   & \vdots \end{array}. $$


\begin{remark} \label{rmk:sdast}
Let $R_i$ denote the $i$th row of Stern's diatomic array, beginning
with row 0. Form the concatenation 
   $$ R_0R_1\cdots R_{n-2}R_{n-1}R_{n-1}R_{n-2}\cdots R_1R_0 $$
and then merge together the last 1 in each row with the first 1 in the
next row. We then obtain row $n$ of Stern's triangle. From this
observation almost any property of Stern's triangle can be carried
over straightforwardly to Stern's diatomic array and \emph{vice
  versa}, including the properties that we discuss below. See in
particular Remark~\ref{rmk:sba}.
\end{remark}

We prefer Stern's triangle to Stern's diatomic array because its
properties are more elegant and simple. In particular, let $\anbc nk$
denote the $k$th entry (beginning with $k=0$) in row $n$ of Stern's
triangle, in analogy with Pascal's triangle. Naturally we define
$\anbc nk=0$ if $k$ is ``out of range,'' i.e., if $k<0$ or
$k>2^{n+1}-2$. We then have the following ``Stern analogue'' of the
binomial theorem.

\begin{theorem} \label{thm:fnx}
Let $n\geq 1$. Then
  \beq F_n(x)\coloneqq \sum_{k\geq 0}\anbc nk x^k = 
      \prod_{i=0}^{n-1}\left( 1+x^{2^i}+x^{2\cdot 2^i}\right). 
    \label{eq:fnx} \eeq
\end{theorem}

\begin{proof}
This formula is immediate from the recursive definition of Stern's
triangle. The result is clearly true for $n=1$, so we need to show
that $F_{n+1}(x)=(1+x+x^2)F_n(x^2)$, $n\geq 1$. The product
$xF_n(x^2)$ accounts for the entries in row $n+1$ that are brought
down from row $n$, while $(1+x^2)F_n(x^2)$ accounts for the entries in
row $n+1$ that are the sum of two consecutive entries in row $n$.
\end{proof}

Note that it follows immediately from Theorem~\ref{thm:fnx} that
$\anbc nk$ has the following combinatorial interpretation: it is the
number of partitions of $k$ (as defined, e.g., in
\cite{andrews} or \cite[{\S}1.8]{ec1}) into the parts $1,2,4,\dots,
2^{n-1}$, where each part may be used at most twice.

Theorem~\ref{thm:fnx} implies that row $n$ of Stern's triangle
approaches a limiting sequence $b_0,b_1,\dots$ as $n\to\infty$,
meaning that for all $k\geq 0$ we have $\anbc nk=b_k$ for $n$
sufficiently large (depending on $k$). Moreover, letting $n\to\infty$
in equation~\eqref{eq:fnx} shows that
  $$ \sum_{k\geq 0} b_kx^k = \prod_{i=0}^\infty 
      \left( 1+x^{2^i}+x^{2\cdot 2^i}\right). $$
The sequence $b_0,b_1,\dots$ is the famous \emph{Stern's diatomic
  sequence}, with many remarkable properties. (Sometimes
$0,b_0,b_1,\dots$ is called Stern's diatomic sequence.) Perhaps the
most amazing property, though irrelevant here, is that every positive
rational number appears exactly once as a ratio $b_k/b_{k+1}$, and
that this fraction is always in lowest terms. Northshield
\cite{northshield} has a nice survey. 

\section{Sums of products of powers.} \label{sec2}
In this section we will consider sums $\sum_{k\geq 0} \anbc
nk^r$, where $r\in\nn=\{0,1,\dots\}$, and more
generally 
 $$ u_\alpha(n)\coloneqq\sum_{k\geq 0}\anbc
   nk^{\alpha_0}\anbc{n}{k+1}^{\alpha_1} 
   \cdots \anbc{n}{k+m-1}^{\alpha_{m-1}}, $$ 
where $\alpha=(\alpha_0,\dots,\alpha_{m-1})\in\nn^m$. For
notational simplicity we write
$u_{\alpha_0,\dots,\alpha_{m-1}}(n)$ as short for
$u_{(\alpha_0,\dots,\alpha_{m-1})}(n)$.

We will use the following terminology concerning linear recurrences
with constant coefficients. Suppose that $f\colon\nn\to\cc$ satisfies
such a recurrence for $n\geq n_0$. Thus
there are complex numbers $c_1,\dots,c_\ell$, with $c_\ell\neq 0$,
such that 
  $$ f(n+\ell)+c_1f(n+\ell-1)+c_2f(n+\ell-2)+\cdots+c_\ell
f(n)=0,\ \ n\geq n_0. $$ 
Equivalently, $f(n)$ has a generating function of the form
  \beq \sum_{n\geq 0}f(n)x^n = \frac{P(x)}{1+c_1x+\cdots+c_\ell
    x^\ell}, \label{eq:ratgf} \eeq
where $P(x)\in\cc[x]$ and $\deg P(x)<\ell+n_0$. We say that $f(n)$ has
a \emph{rational 
  generating function} with \emph{characteristic polynomial}
  $$ R(x)= x^\ell+c_1 x^{\ell-1}+\cdots+c_\ell. $$
Thus if $Q(x)$ denotes the denominator of the right-hand side of
equation~\eqref{eq:ratgf}, then $R(x)=x^\ell Q(1/x)$. For further
information on linear recurrences with constant coefficients and
rational generating functions, see \cite[Chapter~4]{ec1}.

Clearly the number of entries in row $n$ of Stern's triangle is
$2^{n+1}-1$, so $u_0(n)=2^{n+1}-1$. It is also clear that
$u_1(n)=3^n$, e.g., by putting $x=1$ in equation~\eqref{eq:fnx} or
directly from the recursive structure of Stern's triangle, since each
entry in row $n$ contributes to three entries of row $n+1$.

Thus let us turn to $u_2(n)$. The first few values (beginning at
$n=0$) are  1, 3, 13, 59, 269, 1227, 5597, 25531,$\dots$. By various 
methods, such as trial-and-error, using the \emph{Online
  Encyclopedia of Integer Sequences} (OEIS) , or using the Maple
package gfun, we are led to conjecture that
  \beq u_2(n+2) -5u_2(n+1)+2u_2(n)=0,\ \ n\geq 0. \label{eq:u2n} \eeq
Equivalently (using the initial values $u_2(0)=1$ and $u_2(1)=3$), 
   $$ \sum_{n\geq 0} u_2(n)x^n = \frac{1-2x}{1-5x+2x^2}. $$
Note the difference from Pascal's triangle, where $\sum_{k\geq 0}
\binom nk^2 =\binom{2n}{n}$ and $\sum_{n\geq
  0}\binom{2n}{n}x^n=1/\sqrt{1-4x}$. For Stern's triangle the
generating function is rational, while for Pascal's triangle it is
only algebraic.

How do we prove the recurrence~\eqref{eq:u2n}? From the definition of
Stern's triangle we have
  \beas u_2(n+1) & = & \sum_k \left(\anbc nk+\anbc{n}{k+1}\right)^2 +
     \sum_k \anbc nk^2\\ & = &
      3u_2(n) + 2\sum_k \anbc nk\anbc{n}{k+1}\\ & = &
      3u_2(n)+2u_{1,1}(n). \eeas
We therefore need to play a similar game with $u_{1,1}$:
   \beas u_{1,1}(n+1) & = & \sum_k \left(\anbc{n}{k-1}+
     \anbc nk
    \right)\anbc nk\\ & & \ \  + \sum_k \anbc nk\left(\anbc
     nk+\anbc{n}{k+1}\right)\\ & = & 2u_2(n)+2u_{1,1}(n). \eeas
Hence we get the matrix recurrence
    \beq A\left[ \begin{array}{c}u_2(n)\\ u_{1,1}(n)\end{array} \right]
     = \left[ \begin{array}{c}u_2(n+1)\\ u_{1,1}(n+1)\end{array}
       \right], \label{eq:matrec} \eeq
where $A= \left[ \begin{array}{cc} 3 & 2\\ 2 & 2
    \end{array} \right]$.
This is a standard type of simultaneous linear recurrence. To solve
it, we have 
   $$ A^n\left[ \begin{array}{c}u_2(1)\\ u_{1,1}(1)\end{array}
    \right] = \left[ \begin{array}{c}u_2(n)\\ u_{1,1}(n)\end{array}
    \right].  $$      
The minimum polynomial of $A$, i.e., the (nonzero) monic
polynomial $M(x)$ of least degree satisfying $M(A)=0$, is easily
computed to be $x^2-5x+2$.  Hence
 \beas \left[\begin{array}{c} 0\\ 0\end{array} \right] & = & 
   A^n(A^2-5A+2)
    \left[ \begin{array}{c}u_2(1)\\ u_{1,1}(1)\end{array}
    \right]\\ & = & (A^{n+2}-5A^{n+1}+2A^n)
   \left[ \begin{array}{c}u_2(1)\\ u_{1,1}(1)\end{array}
    \right]\\ & = &
   \left[ \begin{array}{c}u_2(n+2)\\ u_{1,1}(n+2)\end{array} 
    \right] -5\left[ \begin{array}{c}u_2(n+1)\\ u_{1,1}(n+1)\end{array}
    \right]+2\left[ \begin{array}{c}u_2(n)\\ u_{1,1}(n)\end{array}
    \right], \eeas
so we get $u_2(n+2)-5u_2(n+1)+u_2(n)=0$, as well as the same recurrence
for $u_{1,1}(n)$.

Let us apply this procedure to $u_3(n)=\sum_k \anbc nk^3$. We get
  \beas u_3(n+1) & = & \sum_k \left(\anbc nk+\anbc{n}{k+1}\right)^3 +
     \sum_k \anbc nk^3\\ & = &
      3u_3(n)+3u_{2,1}(n)+3u_{1,2}(n). \eeas
Because of the symmetry of Stern's triangle about a vertical axis, we
have 
  $$ u_{\alpha_0,\alpha_1,\dots,\alpha_{m-1}}(n)=
    u_{\alpha_{m-1},\dots,\alpha_1,\alpha_0}(n), $$
so in particular $u_{2,1}(n)=u_{1,2}(n)$. Thus 
  $$ u_3(n+1) = 3u_3(n)+6u_{2,1}(n). $$
Similarly,
  \beas u_{2,1}(n+1) & = &  \sum_k \left(\anbc{n}{k-1}+
     \anbc nk
    \right)^2\anbc nk\\ & & \ \  + \sum_k \anbc nk\left(\anbc
     nk+\anbc{n}{k+1}\right)^2\\ & = & 2u_3(n)+4u_{2,1}(n). \eeas
The matrix $\left[ \begin{array}{cc} 3 & 6\\ 2 & 4\end{array}\right]$
has minimum polynomial $x(x-7)$, so we get the recurrence
  $$ u_3(n+1) = 7 u_3(n),\ \ n\geq 1, $$
and similarly for $u_{2,1}(n)$. (The recurrence is not valid at $n=0$
since the minimum polynomial is $x(x-7)$, not $x-7$.) In fact, we have
the surprisingly simple formulas
  $$ u_3(n)=3\cdot 7^{n-1},\ \ u_{2,1}(n)=2\cdot 7^{n-1},\ \ n\geq
   1. $$
Here we see an even larger divergence from the behavior of Pascal's
triangle---the generating function for $f(n)\coloneqq \sum_{k\geq 0}
\binom nk^3$ is not even algebraic. The best we can say is that it is
D-finite \cite[{\S}6.4]{ec2}, meaning that $f(n)$ satisfies a linear
recurrence with polynomial coefficients, namely,
   $$ (n+2)^2f(n+2)-(7n^2+21n+16)f(n+1)-8(n+1)^2f(n) =0,\
  n\geq 0. $$
For more on the sums $\sum_k \binom nk^r$, see
\cite[Exercise~6.54]{ec2} and the references given there.

We have shown that $u_2(n), u_{1,1}(n), u_3(n)$, and $u_{2,1}(n)$ have
rational generating functions. The same technique yields the following
general result.

\begin{theorem} \label{thm:ualpharat}
For any $\alpha=(\alpha_0,\dots,\alpha_{m-1})\in\nn^m$, the function
$u_\alpha(n)$ has a rational generating function.
\end{theorem}

\begin{proof}
We have
  \bea u_\alpha(n+1) & = & \sum_k \anbc nk^{\alpha_0}
     \left( \anbc nk+\anbc{n}{k+1}\right)^{\alpha_1}
     \anbc{n}{k+1}^{\alpha_2}\\ & & \ \ \
     \cdot\left( \anbc{n}{k+1}+\anbc{n}{k+2}\right)^{\alpha_3}
    \cdots \nonumber\\ & & \ \ \
    + \sum_k \left( \anbc nk+\anbc{n}{k+1}\right)^{\alpha_0}
       \anbc{n}{k+1}^{\alpha_1} \nonumber\\ & & \ \ \
      \cdot\left( \anbc{n}{k+1}+\anbc{n}{k+2}\right)^{\alpha_2}
      \anbc{n}{k+2}^{\alpha_3}\cdots. \label{eq:uarec} \eea
When the summands are expanded, we obtain an expression for
$u_\alpha(n+1)$ as a linear combination of $u_\beta$'s. Define the
\emph{spread} of $u_\alpha$, denoted $\spread(u_\alpha)$, to be largest
length (number of terms) of 
any $\beta$ for which $u_\beta(n)$ appears in this expression for
$u_\alpha(n+1)$. For instance, from $u_2(n+1)=3u_2(n)+2u_{1,1}(n)$ we
see that $\spread(u_2)=2$, coming from $\beta=(1,1)$ of length two. 

From equation~\eqref{eq:uarec} we see that that when $u_\alpha(n+1)$
is written as a linear combination of $u_\beta$'s, the indices $\beta$
that occur satisfy (a) $|\alpha|\coloneqq \sum \alpha_i=\sum
\beta_i=|\beta|$, and (b) $\spread(u_\alpha)=2+\lfloor \frac
12(\ell-1)\rfloor$, where $\alpha$ has length 
$\ell$. Since $2+\lfloor \frac 12(\ell-1)\rfloor\leq \ell$ for
$\ell\geq 2$, it 
follows that for $\alpha=(r)$ we will obtain a (finite) matrix
recurrence like 
equation~\eqref{eq:matrec}, where one of the functions is
$u_\alpha(n)$. The size (number of rows or columns) of the matrix is
$1+\lfloor r/2\rfloor$, the number of weakly 
decreasing sequences of positive integers with sum $r$ and length
1 or 2.\ \ Similarly, when $\alpha$ has length $\ell\geq
2$, then the size of the matrix will not exceed the number of
equivalence classes of sequences of nonnegative integers of length at
most $\ell$ summing to $|\alpha|$, where a sequence $\alpha$ is
equivalent to its reverse. The point is that there are only finitely
many such equivalence classes, so we obtain a finite matrix equation. 
By the same argument used to show $u_2(n+2)-5u_2(n+1)+u_2(n)=0$, we
get that $u_\alpha(n)$ has a rational generating function.
\end{proof}

Here are the characteristic polynomials of the recurrences satisfied
by $u_r(n)$ for $n$ sufficiently large (denoted $n\gg 0$), for $1\leq
r\leq 10$: 
  $$ \begin{array}{c} x-3,\\ x^2-5x+2,\\ x-7,\\ (x+1)(x^2-11x+2),\\
     x^2-14x-47,\\ x^4-20x^3-161x^2-40x+4,\\ x^3-29x^2-485x-327,\\
     (x+1)(x^4-44x^3-1313x^2-88x+4),\\ x^3-65x^2-3653x-3843,\\
     (x+1)(x^4-100x^3-9601x^2-200x+4). \end{array} $$

We can say quite a bit more about the recurrence satisfied by
$u_\alpha(n)$. We can assume that
$\alpha=(\alpha_0,\dots,\alpha_{m-1})\in\nn^m$ with $\alpha_0>0$ and
$\alpha_{m-1}>0$. We then write $m=\ell(\alpha)$, the length of
$\alpha$. 

Write $\mmp(\alpha)$ (for ``matrix minimum polynomial'') for the
minimum polynomial $M(A_\alpha)$ of the matrix $A_\alpha$ used to
compute $u_\alpha(n)$ by the method above. Write $\rmp(\alpha)$ (for
``recurrence minimum polynomial'') for the characteristic polynomial
of the linear recurrence with constant coefficients \emph{of least
  degree} satisfied by $u_\alpha(n)$ for $n\gg 0$. Note that the
proof of Theorem~\ref{thm:ualpharat} shows that $\mmp(\alpha)$ is
divisible by $\rmp(\alpha)$.

Linearly order all sequences $\alpha=(\alpha_0,\dots,\alpha_{m-1})$
with $\alpha_0>0$ and $\alpha_{m-1}>0$, and where $|\alpha|=r$ is
fixed, in such a way that the following conditions are satisfied: (a)
if $\ell(\alpha)<\ell(\beta)$ then $\alpha<\beta$; and (b) if
$\alpha=(\alpha_1,\alpha_2,\alpha_3)$ and
$\beta=(\beta_1,\beta_2,\beta_3)$ (with $|\alpha|=|\beta|$), then
define $\alpha\leq \beta$ if $\alpha_1\geq \beta_1$,
$\alpha_2\leq \beta_2$, and $\alpha_3\geq \beta_3$. (We do not specify
the ordering when $\ell(\alpha)=\ell(\beta)\neq 3$ except for condition
(a).)

Order the rows and columns of $A_\alpha$ using the order defined in
the previous paragraph. It is easy to check that $A_\alpha$ is block
lower-triangular. The first block is $A_r$, where $r=|\alpha|$. This
corresponds to rows and columns indexed by $\beta$ with
$\ell(\beta)\leq 2$. For $\ell(\beta)=3$, the blocks are $1\times 1$
with entry 1. All the other blocks are $1\times 1$ with entry 0.

\begin{example}
For $\alpha=(1,1,1,1)$ we use the ordering (writing
e.g.\ 121 for $(1,2,1)$) 
   $$ 4<31<22<121<211<1111. $$
(No other $\beta$'s occur in computing the recurrence satisfied by
$u_{1,1,1,1}$.) We get the matrix
  $$ A_{(1,1,1,1)} =\left[ \begin{array}{cccccc}
     3 & 8 & 6 & 0 & 0 & 0\\ 2 & 5 &3 & 0 & 0 & 0\\
     2 & 4 & 2 & 0 & 0 & 0\\ 1 & 4 & 2 & 1 & 0 & 0\\
     1 & 3 & 1 & 2 & 1 & 0\\ 0 & 2 & 2 & 2 & 2 & 0 \end{array}
     \right]. $$
In particular, the minimum polynomial of $A_{(1,1,1,1)}$ is
$x(x+1)(2x^2-11x+1)(x-1)^2$. The factor $(x+1)(2x^2-11x+1)$ is the
minimum polynomial of the block $A_4$. The characteristic polynomial
$\rmp(1,1,1,1)$ of the least order recurrence satisfied by
$u_{1,1,1,1}(n)$ for $n\gg 0$ turns out to be
$(x-1)^2(x+1)(2x^2-11x+1)$. 
\end{example}

The above argument yields the following theorem.

\begin{theorem}
Let $\alpha\in \nn^m$ and $|\alpha|=r$. Then the polynomial
$\mmp(\alpha)$ has the form $x^{w_\alpha}(x-1)^{z_\alpha}\mmp(r)$
for some $w_\alpha, z_\alpha\in\nn$. 
\end{theorem}

We have not considered whether there is a ``nice'' description of the
integers $w_\alpha$ and $z_\alpha$, nor the largest power of $x-1$
dividing $\rmp(\alpha)$.


\section{A conjecture on the order of the recurrence.}
Can we say more about the actual recurrence satisfied by
$u_\alpha(n)$? We have not investigated this question systematically,
but we do have a conjecture about the order of the recurrence and some
special properties of the characteristic 
polynomial. For instance, is it just an ``accident'' that the matrix
$A_3=\left[ \begin{array}{cc} 3 & 6\\ 2 &
    4\end{array}\right]$ has a zero eigenvalue, thereby reducing the
order of the recurrence from two to one? Or that the
polynomials $\rmp(4)$, $\rmp(8)$, and $\rmp(10)$,  are divisible by
$x+1$? 

We noted before that the matrix $A_r$ has size $\lceil
(r+1)/2\rceil$ (the number of weakly decreasing sequences of
positive integers with sum $r$ and length 1 or 2)
so $u_r(n)$ satisfies a linear recurrence with constant coefficients
of this order. However, on
the basis of empirical evidence ($r\leq 125$), we conjecture that the
least order of such a recurrence is actually $\frac
13r+O(1)$. In fact, we have the following more precise
conjecture. Write $[a_0,\dots,a_{q-1}]_q$ for the periodic function
$f\colon \nn\to\rr$ satisfying $f(n)=a_i$ if $n\equiv
i\modd{q}$.  Let $e_r(\theta)$ denote the number of eigenvalues of
$A_r$ equal to $\theta$.  Recall that an eigenvalue $\theta$ of a
matrix $A$ is \emph{semisimple} if the minimum polynomial of
$A$ is not divisible by $(x-\theta)^2$. Equivalently, all the Jordan
blocks with eigenvalue $\theta$ of the Jordan canonical form of $A$
have size one. 

\begin{conjecture} \label{conj:order}
\be\item[(a)] We have
  $$ e_{2s-1}(0) = \frac 13 s+\left[0,-\frac 13,\frac 13\right]_3, $$
 and the eigenvalue 0 is semisimple. There are no other multiple
 eigenvalues, and 1 is not an eigenvalue. 
  \item[(b)] We have
   \beas e_{2s}(1) & = & \frac 16s +\left[-1,-\frac 16, -\frac 13, -\frac
     12, -\frac 23, \frac 16\right]_6\\
   e_{2s}(-1) & = & e_{2s+6}(1). \eeas
The eigenvalues $1$ and $-1$ are semisimple, and there are no other
multiple eigenvalues. 
 \ee
\end{conjecture}


Let $\mo(r)$ be the minimum order of a linear recurrence with
constant coefficients satisfied by $u_r(n)$ for $n\gg 0$. 
Conjecture~\ref{conj:order} reduces the ``naive'' bound
$\mo(r)\leq \lceil r/2\rceil$ to $\mo(r)\leq \lceil
r/2\rceil-e_r(0)$ when $r$ is odd, and 
  $$ \mo(r) \leq \left\lceil \frac r2\right\rceil  
     - \max\{0,e_r(1)-1\}- \max\{0,e_r(-1)-1\} $$
when $r$ is even. However, it appears that the eigenvalue 1 of
$A_{2s}$ is always superfluous, that is, $x-1$ is not a factor of the
characteristic polynomial $\rmp_{2s}(x)$. This will lower the upper
bound for $\mo(r)$ by 1 when $e_r(1)>0$. The evidence suggests that we
then get a best possible result. The resulting conjecture is the
following. 

\begin{conjecture}
We have $\mo(2)=2$, $\mo(6)=4$, and otherwise
    \beas \mo(2s) & = & 2\left\lfloor \frac s3\right\rfloor+3,\ \
    s\neq 1,3\\ \mo(6s+1) & = & 2s+1,\ \ s\geq 0\\ \mo(6s+3) & = &
    2s+1,\ \ s\geq 0\\ \mo(6s+5) & = & 2s+2,\ \ s\geq 0. \eeas
\end{conjecture}

After the above conjectures were communicated in a lecture, David
Speyer \cite{speyer} made some important progress. He showed that the
conjectured values of $e_{2s-1}(0)$ and $e_{2s}(\pm 1)$ are lower
bounds for their actual values. 
Moreover, $A_r$ can be conjugated by a diagonal matrix to give a
symmetric matrix, thereby showing that the eigenvalues of $A_r$ are
semisimple (and real). As a consequence, the conjectured
value of $\mo(r)$ is an upper bound on its actual value. The key to
Speyer's argument is that if $B_r$ is defined like the matrix $A_r$
except that we don't take into account the symmetry
$u_{\alpha_0,\alpha_1,\dots,\alpha_{m-1}}(n) =
  u_{\alpha_{m-1},\dots,\alpha_1,\alpha_0}(n)$, then $B_r$ is the
matrix of the linear transformation $\phi\colon V\to V$, with
respect to the basis of monomials, defined by 
  $$ \phi(f)(x,y) =  f(x+y,y)+f(x,x+y), $$
where $V$ is the vector space of complex homogeneous polynomials of
degree $r$ in the two variables $x$ and $y$. We will not give further
details here.  

\begin{remark} \label{rmk:sba}
Let $v_\alpha(n)$ be the analogue for Stern's diatomic array of
$u_\alpha(n)$. That is, if $\qbc{n}{k}$ denotes the $k$th entry
(beginning with $k=0$) in row $n$ (beginning with $n=0$) in Stern's
diatomic array, then
  $$ v_\alpha(n)\coloneqq\sum_{k\geq 0}\qbc
   nk^{\alpha_0}\qbc{n}{k+1}^{\alpha_1} 
   \cdots \qbc{n}{k+m-1}^{\alpha_{m-1}}, $$ 
where $\alpha=(\alpha_0,\dots,\alpha_{m-1})\in\nn^m$. 

Write $U_\alpha(x)=\sum_{n\geq 0} u_\alpha(n)x^n$ and $V_\alpha(x)=
\sum_{n\geq 0} v_\alpha(n)x^n$. It follows from Remark~\ref{rmk:sdast}
that 
  \beq  \frac{2V_r(x)}{1-x} = \frac{U_r(x)-1}{x} +
    \frac{1+x}{(1-x)^2}. \label{eq:vrur} \eeq
Write $R_\alpha(x)$ for the characteristic polynomial of the linear
recurrence with constant coefficients \emph{of least degree} satisfied
by $v_\alpha(n)$ for $n\gg 0$. If the empirical observation above,
that $\rmp_r(x)$ is not divisible by $x-1$, holds, then it follows
from equation~\eqref{eq:vrur} that $R_r(x)=(x-1)\rmp_r(x)$. 
\end{remark}

\section{A generalization.}
A much more general result can be proved by exactly the same
method. Let $p(x)$ and $q(x)$ be any complex polynomials, and let
$b\geq 2$ be an integer. Define
  $$ F_{p,q,b,n}(x) = q(x)\prod_{i=0}^{n-1}p(x^{b^i}). $$
Let $\alpha=(\alpha_0,\dots,\alpha_{m-1})\in\nn^m$. If
  $$ F_{p,q,b,n}(x) = \sum_{i\geq 0} c_i(n)x^i, $$
then set
  $$ u_{p,q,b,\alpha}(n) = \sum_k c_k(n)^{\alpha_0}c_{k+1}(n)^{\alpha_1}
    \cdots c_{k+m-1}(n)^{\alpha_{m-1}}. $$

\begin{theorem} \label{thm:main}
 For fixed $p,q,b$ and $\alpha$, the function $u_{p,q,b,\alpha}(n)$
 has a rational generating function.
\end{theorem}

\begin{proof}
Just as in the previous section we can express $u_{p,q,b,\alpha}(n+1)$
as a linear combination of $u_{p,q,b,\beta}(n)$'s. If
$u_{p,q,b,\beta}(n)$ actually appears (i.e., has a nonzero
coefficient), then call $\beta$ a \emph{child} of $\alpha$. Successive
children of $\alpha$ are called \emph{descendants} of $\alpha$.  The
only issue is whether $\alpha$ has only finitely many descendants. It
is clear that all the descendants $\gamma$ satisfy
$|\alpha|=|\gamma|$. 


In the previous section we observed that if $\alpha$ has length $\ell$,
then $\spread(u_\alpha)=1+\lceil \ell/2\rceil$, so that we can take
$s=2$. In the present situation, we can assume that $p(0)\neq 0$ and
$q(0)\neq 0$. Let $h=\deg p+\deg q$. If $\alpha$ has length $\ell$,
then
  $$ \spread(u_{p,q,b,\beta}) = 1+\left\lfloor \frac
    hb\right\rfloor 
      +\left\lfloor \frac{\ell-1}{b}\right\rfloor. $$
(The precise formula is irrelevant. One just needs to see that if
$\ell$ increases by $b$ then the spread increases by 1.) Since
$b\geq 2$, for sufficiently large $\ell$ we will have
  \beq \spread(u_{p,q,b,\beta})\leq \ell, \label{eq:spread}
  \eeq 
showing that $\alpha$ has finitely many descendants.
\end{proof}

Note that the above proof breaks down for Pascal's triangle, as it
should. For then $b=1$, so the inequality~\eqref{eq:spread} does not
hold for sufficiently large $\ell$.

\begin{remark}
There is a straightforward multivariate generalization of
Theorem~\ref{thm:main}. The polynomials $p(x)$ and $q(x)$ are replaced
by complex multivariate polynomials $p(x_1,\dots,x_d)$ and
$q(x_1,\dots,x_d)$, and $b$ is replaced by $d$ integers
$b_1,\dots,b_d\geq 2$. We define
  $$ F_{p,q,b,n}(x) = q(x_1,\dots,x_d)\prod_{i=0}^{n-1}
    p(x_1^{b_1^i},\dots,x_d^{b_d^i}), $$
and the development proceeds as before. Details are omitted. As some
random examples, extend the definition of $\rmp(\alpha)$ to
$\rmp(p,q,\alpha,b)$, where $b=(b_1,\dots,b_d)$. Then
  \beas 
   \rmp((1+x_1+x_2)^2,1,(2),(2)) & = & x^2-27x+132,\\
   \rmp((1+x_1+x_2)^2,1,(3),(2)) & = & x^3-67x+1020x^2-4704,\\
   \rmp((1+x_1+x_2)^2,1,(2),(2,3)) & = &  x^2-23x+104,\\
   \rmp((1+x_1+x_2)^2,1,(3),(2,3)) & = &  x^2-45x+402,\\
   \rmp((1+x_1+x_2)^2,1,(4),(2,3)) & = & x^3-107x^2+3176x-28320. \eeas
\end{remark}

Compare with the univariate analogue $p(x)=(1+x)^d$, where for
instance 
   \beas \rmp((1+x)^2,1,(2),2) & = & (x-2)(x-8),\\ 
    \rmp((1+x)^2,1,(3),2) & = &  (x-4)(x-16),\\
    \rmp((1+x)^2,1,(4),2) & = & (x-2)(x-8)(x-32),\\
    \rmp((1+x)^3,1,(2),2) & = & (x-2)(x-8)(x-32),\\ 
    \rmp((1+x)^3,1,(3),2) & = & (x-2)(x-8)(x-32)(x-128),\\
    \rmp((1+x)^3,1,(4),2) & = & (x-2)(x-8)(x-32)(x-128)(x-512).
   \eeas
The reason for this nice factorization is discussed in the
next section. 

\section{A special case.}
A natural problem is to say more about the recurrence (or equivalently
its characteristic polynomial) satisfied by $u_{p,q,b,\alpha}$ in
general, or at least in special situations. In this section we give
one such result.

For simplicity we first consider the case $p(x)=(1+x)^3$, $q(x)=1$,
$b=2$, and $\alpha=(r)$. We then state a more general result that is
proved by exactly the same technique.

\begin{theorem}
For $r\geq 1$ we have
  $$ u_{(1+x)^3,1,2,r}(n)=\sum_{i=0}^r
     c_i 2^{(2i+1)n} $$
for certain rational constants $c_i$ (depending on $r$).
\end{theorem}

\begin{proof}
The key to the proof is the simple and well-known identity
  $$ (1+x)(1+x^2)(1+x^4)\cdots\left(1+x^{2^{t-1}}\right)
     =\frac{1-x^{2^t}}{1-x}. $$
Thus
  \beas \prod_{j=0}^{t-1}\left( 1+x^{2^j}\right)^3 & = &
     \frac{\left(1-x^{2^t}\right)^3}{(1-x)^3}\\ & = &
     \frac{1-3x^{2^t}+3x^{2\cdot 2^t}-x^{3\cdot 2^t}}
      {(1-x)^3}. \eeas
Let us consider more generally the generating function
  \beas H_m(x) & = & \frac{(1-x^m)^3}{(1-x)^3}\\ & = &   
    \frac{1-3x^m+3x^{2m}-x^{3m}}{(1-x)^3}. \eeas
Now 
   $$ \frac{x^p}{(1-x)^3} = 
     \sum_{k\geq 0} \binom{k+2}{2}x^{p+k}. $$
Since $H_m(x)$ is a polynomial in $x$ of degree $3(m-1)$, we get
  \beas H_m(x) & = & \sum_{k=0}^m \binom{k+2}{2}x^k +\sum_{k=m}^{2m-1}
   \left[\binom{k+2}{2}-3\binom{k-m+2}{2}\right]x^k\\ & & \ \ \
    +\sum_{k=2m}^{3m-1}\left[ \binom{k+2}{2}-3\binom{k-m+2}{2}
      +3\binom{k-2m+2}{2}\right]x^k. \eeas
Let 
  \bea P(m) & = & \sum_{k=0}^m \binom{k+2}{2}^r
  +\sum_{k=m}^{2m-1} 
   \left[\binom{k+2}{2}-3\binom{k-m+2}{2}\right]^r \nonumber\\ & & \ \ \
    + \sum_{k=2m}^{3m-1}\left[ \binom{k+2}{2}-3\binom{k-m+2}{2}
      +3\binom{k-2m+2}{2}\right]^r. \label{eq:pm3terms} \eea
The binomial coefficient $\binom{k-jm+2}{2}$ is a polynomial in $m$ of
degree two. For any polynomial $Q(m)$ of degree $d$, the sum
$\sum_{k=0}^m Q(k)$ is a polynomial of degree $d+1$, so the same is
true of $\sum_{k=m}^{2m-1}Q(k)$, etc. Hence $P(m)$ is a polynomial of
degree at most $2r+1$, say $P(m)=\sum_{i=0}^{2r+1}a_im^i$. Then,
  \beas  u_{(1+x)^3,1,2,r}(n) & = & P(2^n)\\ & = &
    \sum_{i=0}^{2r+1}a_i2^{ni}. \eeas

It remains to prove that $a_i=0$ if $i$ is even. The coefficient of
$x^k$ in $H_m(x)$ is 0 for $k>3(m-1)$, so
  \beq \binom{k+2}{2}-3\binom{k-m+2}{2}
      +3\binom{k-2m+2}{2} -\binom{k-3m+2}{2} = 0 \label{eq:0coef}
  \eeq
for $k>3(m-1)$. The left-hand side is a polynomial in $k$ and
$m$. Since it is 0 for $k>3(m-1)$, it must be 0 as a polynomial in $k$
and $m$. 

In general, for integers $a<b<c$ and any function $f(i)$, we have
  \beq \sum_{i=a}^b f(i) + \sum_{i=b+1}^c f(i) = \sum_{i=a}^c f(i). 
    \label{eq:sumsum} \eeq
If we want to define $\sum_{i=a}^bf(i)$ for integers $a>b$ so that
equation~\eqref{eq:sumsum} is valid for \emph{all} integers $a,b,c$,
then we will have the identity
  $$ \sum_{i=a}^b f(i) = -\sum_{b+1}^{a-1} f(i). $$
It is easy to check (using the fact that if two univariate complex
polynomials agree for infinitely many values, then they are the same
polynomial) that if 
  $$ G(m)=\sum_{i=A(m)}^{B(m)} Q(m) $$ 
for polynomials $Q,A,B$ (so $G$ is also a polynomial), then indeed we
have 
  $$ G(-m) = \sum_{i=A(-m)}^{B(-m)}Q(-m). $$

Write the polynomial $P(m)$ of equation~\eqref{eq:pm3terms} as
$P_1(m)+P_2(m)+P_3(m)$, corresponding to the three sums. 
If we substitute $-m$ for $m$ in 
  $$ P_3(m) =\sum_{k=2m}^{3m-1}\left[ \binom{k+2}{2}-3\binom{k-m+2}{2}
       +3\binom{k-2m+2}{2}\right]^r, $$
then we get 
  \beas P_3(-m) & = & \sum_{k=-2m}^{-3m+1} 
     \left[ \binom{k+2}{2}-3\binom{k+m+2}{2}
       +3\binom{k+2m+2}{2}\right]^r\\ & = & -\sum_{k=-3m}^{-2m-1}
     \left[ \binom{k+2}{2}-3\binom{k+m+2}{2}
       +3\binom{k+2m+2}{2}\right]^r\\ & = &
     -\sum_{k=0}^{m-1} \left[ \binom{k-3m+2}{2}-3\binom{k-2m+2}{2}
       +3\binom{k-m+2}{2}\right]^r. \eeas
From equation~\eqref{eq:0coef} there follows 
    \beas P_3(-m) & = & -\sum_{k=0}^{m-1} \binom{k+2}{2}^r\\ & = & -P_1(m).
    \eeas
Hence $P_1(m)+P_3(m)$ is an odd polynomial, i.e.,
   $$ P_1(-m)+P_3(-m)=-(P_1(m)+P_3(m)), $$
so all powers of $m$ appearing in this polynomial have odd exponents.

In exactly the same way, this time using equation~\eqref{eq:0coef} in
the form
  $$ \binom{k+2}{2}-3\binom{k-m+2}{2} =
      -3\binom{k-2m+2}{2} +\binom{k-3m+2}{2}, $$
we obtain that $P_2(-m)=-P_2(m)$. Thus $P(-m)=-P(m)$, completing the
proof. 
\end{proof}

The reader who has followed the previous proof should have no trouble
extending it to the following more general result. We only point out
one possible subtlety: when $d$ is even in the theorem below, the
polynomial $(1-x^b)^d$ has an odd number of terms. Hence the analogue
of the equation $P_i(-m)= -P_i(m)$ becomes $P_i(-m)=-(-1)^r
P_i(m)$. Thus $P_i(m)$ is an even polynomial when $d$ is even and $r$
is odd.

\begin{theorem} \label{thm:nice}
 \be \item[(a)] Let $b\geq 2$, $d\geq 1$, and
 $p(x)=(1+x+\cdots+x^{b-1})^d$. For 
  any $\alpha\in \nn^m$ and $q(x)\in\cc[x]$ we have
  $$ u_{p,q,b,\alpha}(n) = \sum_{i=0}^{1+(d-1)|\alpha|} c_i b^{in}, $$
where $c_i\in\cc$.
 \item[(b)] If $q(x)=1$, $\alpha=(r)$ and either $r$ is even or $d$ is
   odd, then $c_i=0$ when $i$ is even.
 \item[(c)] If $q(x)=1$, $\alpha=(r)$, $r$ is odd, and $d$ is even,
   then $c_i=0$ when $i$ is odd.
 \ee 
\end{theorem}

We leave as an open problem to see to what extent
Theorem~\ref{thm:nice} can be generalized.



\end{document}